\documentclass[preprint]{elsarticle}

\usepackage{amssymb,amsmath,amsthm, color}
\usepackage{inputenc}
\usepackage{url}
\usepackage{amsmath}
\newtheorem{theorem}{Theorem}[section]
\newtheorem{lemma}[theorem]{Lemma}
\newtheorem{define}[theorem]{Definition}
\newtheorem{cor}[theorem]{Corollary}
\newtheorem{prop}[theorem]{Proposition}

\newcommand{\n}{\mathcal N}
\newcommand{\I}{\mathcal I}
\newcommand{\F}{\mathbb F}

\newcommand{\N}{\mathbb N}
\newcommand{\Z}{\mathbb Z}
\newcommand{\C}{\mathcal C}

\newcommand{\GL}{\mathrm{GL}}
\newcommand{\PGL}{\mathrm{PGL}}

\newcommand{\ord}{\mathrm{ord}}

\newcommand{\doublespace}
\begin{document}

\begin{frontmatter}

\title{On the existence and number of invariant polynomials}

\author{Lucas Reis}
\ead{lucasreismat@gmail.com}
%\address{Departamento de Matem\'{a}tica, Universidade Federal de Minas Gerais, Belo Horizonte, MG, 30123-970, Brazil.}
\address{Universidade de S\~{a}o Paulo, Instituto de Ci\^{e}ncias Matem\'{a}ticas e de Computa\c{c}\~{a}o, S\~{a}o
Carlos, SP 13560-970, Brazil.}
\journal{Elsevier}
\begin{abstract}
This paper explores a natural action of the group $\PGL_2(\F_q)$ on the set of monic irreducible polynomials of degree at least two over a finite field $\F_q$. Our main results deal with the existence and number of fixed points and, in particular, we provide some improvements of previous works.

\end{abstract}

\end{frontmatter}
\section{Introduction}
Let $\F_q$ be the finite field with $q$ elements, where $q$ is a power of a prime $p$. We have the following transformations on the polynomial ring $\F_q[x]$. 
 
\begin{define}\label{def:mobius-action}
For $A\in \GL_2(\F_q)$ with $A=\left(\begin{matrix}
a&b\\
c&d
\end{matrix}\right)$ and $f\in \F_q[x]$ a polynomial of degree $k$, set
$$A\circ f=(bx+d)^kf\left(\frac{ax+c}{bx+d}\right).$$
Additionally, if $[A]$ denotes the class of $A$ in the group $\PGL_2(\F_q)$, for $f\in \F_q[x]$ a nonzero polynomial, set
$$[A]\circ f=c_{f, A}\cdot (A\circ f),$$
where $c_{f, A}\in \F_q^*$ is the element of $\F_q$ such that $c_{f, A}\cdot (A\circ f)$ is monic.
\end{define}
Let $\I_k$ be the set of monic irreducible polynomials of degree $k$ over $\F_q$. As pointed out in~\cite{ST12}, the group $\PGL_2(\F_q)$ \emph{acts} on the sets $\I_k$ with $k\ge 2$, via the compositions $[A]\circ f$. It is then natural to ask about the fixed points.
\begin{define}
Let $k\ge 2, f\in \I_k$, $[A]\in \PGL_2(\F_q)$ and $G$ a subgroup of $\PGL_2(\F_q)$. 
\begin{enumerate}[(i)]
\item $f$ is $[A]$-invariant if $[A]\circ f=f$;
\item $f$ is $G$-invariant if $[B]\circ f=f$ for any $[B]\in G$.
\end{enumerate}
\end{define}
From the previous definition, some natural questions arise:
\begin{itemize}
\item  Given a subgroup $G$ of $\PGL_2(\F_q)$, there exists $G$-invariants? 
\item  Given $n\ge 2$ and $[A]\in \PGL_2(\F_q)$, there exists $[A]$-invariants of degree $n$? How many are they?
\end{itemize}
In this paper, we deal with the two questions above and our main results can be stated as follows.
\begin{theorem}\label{thm:main-1}
Let $G$ be a noncyclic group of $\PGL_2(\F_q)$. Then any $G$-invariant has degree two.
\end{theorem}
\begin{theorem}\label{thm:main-2}
Let $[A]\in \PGL_2(\F_q)$ be an element of order $D=\ord([A])$. Then, for any integer $n>2$, the number $\n_{A}(n)$ of $[A]$-invariants of degree $n$ is \emph{zero} if $n$ is not divisible by $D$ and, for $n=Dm$ with $m\in \N$, the following holds:
\begin{equation}\label{eq:invariants-formula}\n_{A}(Dm)=\frac{\varphi(D)}{Dm}\left(c_{A}+\sum_{d|m\atop \gcd(d, D)=1}\mu(d)(q^{m/d}+\eta_{A}(m/d))\right),\end{equation}
where $\varphi$ is the Euler Phi function, $\mu$ is the Mobius function, $\eta_{A}:\N\to \N$ and $c_{A}\in \Z$ are given as follows
\begin{enumerate}
\item $c_{A}=0$ and $\eta_{A}\equiv \varepsilon$, where $\varepsilon=-1$ or $0$, according to whether $A$ has distinct or equal eigenvalues in $\F_q$, respectively; 
\item $c_{A}=-1$ and $\eta_{A}$ is the zero function if $A$ has symmetric eigenvalues in $\F_{q^2}\setminus \F_q$;
\item $c_{A}=0$ and $\eta_{A}(t)=(-1)^{t+1}$ if $A$ has non symmetric eigenvalues in $\F_{q^2}\setminus \F_q$.
\end{enumerate}
\end{theorem}
We remark that some results in the direction of Theorems~\ref{thm:main-1} and~\ref{thm:main-2} were previously obtained. In~\cite{ST12} and~\cite{R17}, Theorem~\ref{thm:main-1} is proved for the cases that $G=\PGL_2(\F_q)$ and $G$ is a $p$-group, respectively. In addition, Theorem 5.3 of~\cite{ST12} entails that, if $[A]\in \PGL_2(\F_q)$ has order $D$ and $n>2$, the number $\n_{A}(n)$ of $[A]$-invariants of degree $n$ equals zero if $n$ is not divisible by $D$ and, for $n=Dm$ with $m\in \N$,
$$\n_{A}(n)\approx \frac{\varphi(D)}{Dm}q^m.$$
Here, $a_m\approx b_m$ means $\lim\limits_{m\to \infty}\frac{a_m}{b_m}=1$. We observe that this asymptotic formula agrees with Theorem~\ref{thm:main-2}.

The structure of the paper is given as follows. In Section 2, we provide all the machinery that is used in the proof of our main results. Section 3 is devoted to prove Theorem~\ref{thm:main-1} and, in Section 4, we prove Theorem~\ref{thm:main-2}.

\section{Preliminaries}
In this section, we provide background material that is frequently used throughout the paper.

\subsection{Auxiliary lemmas}\label{subsec:lemma-aux}
We present, without proof, some auxiliary results from~\cite{ST12}. We use slightly different notations and, for more details, see Sections 4 and 5 of \cite{ST12}.

\begin{lemma}\label{lem:aux-mobius-action}
For any $A, B\in \GL_2(\F_q)$ and $f\in \I_k$ with $k\ge 2$, the following hold.
\begin{enumerate}[(i)]
\item $[A]\circ f$ is in $\I_k$,
\item $[A]\circ ([B]\circ f)=[AB]\circ f$,
\item if $[I]$ is the identity of $\PGL_2(\F_q)$, $[I]\circ f=f$.
\end{enumerate}
\end{lemma}

\begin{define}\label{def:F_Ar}
For $A = \left( \begin{array}{cc}a&b\\ c&d\end{array}\right) \in\GL_2(\F_q)$ and $r$ a non-negative integer,
\[
F_{A,r}(x) := bx^{q^r+1}-ax^{q^r}+dx-c.
\]
\end{define}
%Also, for $A\in \GL_2(\F_q)$, write $A^j= \left( \begin{array}{cc}a_j&b_j\\ c_j&d_j\end{array}\right), j\in \mathbb N$. 
From Theorem~4.5 of~\cite{ST12}, we have the following result.
\begin{lemma}\label{ST4.5}
Let $f$ be an irreducible polynomial of degree $Dm\ge 3$ such that $[A]\circ f=f$, where $D$ is the order of $[A]$. The following hold:
\begin{enumerate}[(i)]
\item there is a unique positive integer $\ell\le D-1$ such that $\gcd(\ell, D)=1$ and $f$ divides $F_{A, s}(x)$, where $s=\ell\cdot \frac{Dm}{D}=\ell\cdot m$,
\item for any $r\ge 1$, the irreducible factors of $F_{A, r}$ are of degree $Dr$, of degree $Dk$ with $k<r$, $r=km$ and $\gcd(m, D)=1$ and of degree at most $2$.
\end{enumerate}
\end{lemma}

\begin{lemma}[see \cite{ST12}, item (a) of Lemma 5.1]\label{power}
Let $r\ge 1$ and let $k$ be a divisor of $r$ such that $m:=r/k$ is relatively prime with $D$, the order of $[A]$. For $j$ such that $jm\equiv 1\pmod D$, the irreducible factors of $F_{A, r}(x)$ of degree $Dk$ are exactly the irreducible factors of $F_{A^j, k}(x)$ of degree $Dk$.
\end{lemma}

\subsection{Invariants through conjugacy classes}
We establish some interesting relations between the polynomials that are invariant by two conjugated elements. We start with the following result.

\begin{lemma}\label{lem:conj-inv}
Let $A, B, P\in \GL_2(\F_q)$ such that $[B]=[P]\cdot [A]\cdot [P]^{-1}$. For any $k\ge 2$ and any $f\in \I_k$, we have that $[B]\circ f=f$ if and only if $[A]\circ g=g$, where $g=[P]^{-1}\circ f\in \I_k$. 
\end{lemma}
\begin{proof}
Observe that, from Lemma~\ref{lem:aux-mobius-action}, the following are equivalent:
\begin{enumerate}[$\bullet$]
\item $[B]\circ f=f,$
\item $[P]\circ ([A]\circ ([P]^{-1}\circ f))=f,$
\item $[A]\circ ([P]^{-1}\circ f)=[P]^{-1}\circ f$.
\end{enumerate}
\end{proof}
\begin{define}
For $A\in \GL_2(\F_q)$, $\C_A(n)$ is the set of $[A]$-invariants of degree $n$ and $\C_A=\cup_{n\ge 2}\C_A(n)$ is the set of $[A]$-invariants.
\end{define}

Since the compositions $[A]\circ f$ preserve degree and permute the set of monic irreducible polynomials of a given degree, we obtain the following result.
\begin{theorem}\label{thm:conjugates-main}
Let $A, B, P\in \GL_2(\F_q)$ such that $B=PAP^{-1}$. Let $\tau:\C_B\to \C_A$ be the map given by $\tau(f)=[P]^{-1}\circ f$. Then $\tau$ is a degree preserving one to one correspondence. Additionally, for any $n\ge 2$, the restriction of $\tau$ to the set $\C_A(n)$ is an one to one correspondence between $\C_A(n)$ and $\C_B(n)$. In particular, $\n_{A}(n)=\n_B(n)$.
\end{theorem}
\begin{proof}
From Lemma~\ref{lem:conj-inv}, $\tau$ is well defined and is a one to one correspondence. Additionally, since the compositions $[A]\circ f$ preserve degree, the restriction of $\tau$ to the set $\C_A(n)$ is a one to one correspondence between $\C_A(n)$ and $\C_B(n)$ and, since these sets are finite, they have the same cardinality, i.e., $\n_{A}(n)=\n_B(n)$.
\end{proof}
We observe that, if $G$ is the cyclic group generated by $[A]\in \PGL_2(\F_q)$, $f$ is $[A]$-invariant if and only if is $G$-invariant. From the previous theorem, the following corollary is straightforward.

\begin{cor}\label{cor:conjugates-main}
Let $G, H\in \PGL_2(\F_q)$ be groups with the property that there exists $P\in \GL_2(\F_q)$ such that $G=[P]\cdot H\cdot [P]^{-1}=\{[P]\cdot [A]\cdot [P]^{-1}\,|\, [A]\in H\}$. Then there is a one to one correspondence between the $G$-invariants and the $H$-invariants that is degree preserving.
\end{cor}

The previous results entail that in order to count polynomials that are invariant by an element (or a subgroup) of $\PGL_2(\F_q)$, we only need to consider them up to conjugations. In this context, understanding the conjugacy classes of the elements in $\PGL_2(\F_q)$ is crucial. 

\begin{define}
For $A\in \GL_2(\F_q)$ such that $[A]\ne [I]$, $A$ is of {\bf type} $1$ (resp. $2$, $3$ or $4$) if its eigenvalues are distinct and in $\F_q$ (resp. equal and in $\F_q$, symmetric and in $\F_{q^2}\setminus \F_q$ or not symmetric and in $\F_{q^2}\setminus \F_q$). 
\end{define}

We observe that he types of $A$ and $\lambda\cdot A$ are the same for any $\lambda\in \F_q^*$. For this reason, we say that $[A]$ is of type $t$ if $A$ is of type $t$. The following theorem entails that any element $[A]\in \PGL_2(\F_q)$ is conjugated to a special element of type $t$, for some $1\le t\le 4$.

\begin{theorem}\label{thm:types}
Let $A\in \GL_2(\F_q)$ such that $[A]\ne [I]$ and let $a, b$ and $c$ be elements of $\F_q^*$. Let $A(a):=\left( \begin{array}{cc}a&0\\ 0&1\end{array}\right),\, \mathcal E:=\left( \begin{array}{cc}1&0\\ 1&1\end{array}\right),\, C(b):=\left( \begin{array}{cc}0&1\\ b&0\end{array}\right)$ and $D(c):=\left( \begin{array}{cc}0&1\\ c&1\end{array}\right)$. Then $[A]\in \PGL_2(\F_q)$ is conjugated to:
\begin{enumerate}[(i)]
\item $[A(a)]$ for some $a\in \F_q\setminus\{0, 1\}$ if and only if $A$ is of type $1$.
\item $[\mathcal E]$ if and only if $A$ is of type $2$.
\item $[C(b)]$ for some non square $b\in \F_q^*$ if and only if $A$ is of type $3$.
\item $[D(c)]$ for some $c\in \F_q$ such that $x^2-x-c\in \F_q[x]$ is irreducible if and only if $A$ is of type $4$.
\end{enumerate}
\end{theorem}
\begin{proof}
This results follows from the fact that two elements  $[A], [B]$ in $\GL_2(\F_q)$ are conjugated if and only if the characteristic polynomials $P_A(x)$ and $P_B(x)$ of $A$ and $B$ are equal up to a transformation of the form $f(x)\mapsto \lambda^{-2}f(\lambda x)$ for some $\lambda\in \F_q^*$. We omit the details.
\end{proof}

\begin{define}
An element $A\in \GL_2(\F_q)$ is in \textbf{reduced form} if it is equal to $A(a)$, $\mathcal E$, $C(b)$ or $D(c)$ for some suitable $a$, $b$ or $c$ in $\F_q$.
\end{define}

We finish this section giving a complete study on the order of the elements in $\PGL_2(\F_q)$, according to their type.  

\begin{lemma}\label{lem:order}
Let $[A]$ be an element of type $t$ and let $D$ be its order in $\PGL_2(\F_q)$. The following hold:
\begin{enumerate}[(i)]
\item for $t=1$, $D>1$ is a divisor of $q-1$,
\item for $t=2$, $D=p$,
\item for $t=3$, $D=2$,
\item for $t=4$, $D$ divides $q+1$ and $D>2$. 
\end{enumerate}
\end{lemma}
\begin{proof}
Since any two conjugated elements in $\PGL_2(\F_q)$ have the same order, from Theorem~\ref{thm:types}, we can suppose that $A$ is in the reduced form. From this fact, items (i), (ii) and (iii) are straightforward. For item (iv), let $\alpha, \alpha^q$ be the eigenvalues of $A$. We that $[A]^D=[A^D]$ and then $[A]^D=[I]$ if and only if $A^D$ equals the identity element $I\in\GL_2(\F_q)$ times a constant. The latter holds if and only if $\alpha^D$ and $\alpha^{qD}$ are equal. Observe that $\alpha^D=\alpha^{qD}$ if and only if $\alpha^{(q-1)D}=1$. In particular, since $\alpha\in \F_{q^2}\setminus \F_q$, we have $D>1$ and $D$ divides $q+1$. If $D=2$, then $\alpha^{q}=-\alpha$, a contradiction since $A$ is not of type $3$.
\end{proof}

\section{On $G$-invariants: the noncyclic case}
Here we provide the proof of Theorem~\ref{thm:main-1}, showing the triviality of $G$-invariants when $G$ is a noncyclic subgroup of $\PGL_2(\F_q)$. We start with the following definition.

\begin{define}
Let $\overline{\F}_q$ be the algebraic closure of $\F_q$. For $A=\left( \begin{array}{cc}a&b\\ c&d\end{array}\right)\in \GL_2(\F_q)$ and $\alpha\in \overline{\F}_q\setminus \F_q$,
$$[A]\circ \alpha:=\frac{d\alpha-c}{a-b\alpha}.$$
\end{define}

According to Lemma 2.6 and Theorem 4.2 of~\cite{ST12}, we have the following result.

\begin{lemma}\label{lem:aux-10}
Let $f\in \F_q[x]$ be a monic irreducible polynomial of degree at least two and $\alpha\in \overline{\F}_q\setminus \F_q$. Let $A, B\in \GL_2(\F_q)$, the following holds:

\begin{enumerate}[(i)]
\item $f$ if $[A]$-invariant if and only if $f$ divides $F_{A, r}$ for some $r\ge 0$ or, equivalently, $[A]\circ \alpha=\alpha^{q^r}$;
\item $[A]\circ ([B]\circ \alpha)=[AB]\circ \alpha$.
\end{enumerate}
\end{lemma}

From the previous lemma, we have the following result.

\begin{lemma}\label{lem:aux-divisor}
Let $r$ be a non-negative integer and $A_1, A_2\in \GL_2(\F_q)$ such that $[A_1]\ne [A_2]$ in $\PGL_2(\F_q)$. If $\alpha\in \overline{\F}_q\setminus \F_q$ is such that $[A_1]\circ \alpha=[A_2]\circ \alpha$, then $\alpha\in \F_{q^2}$.
\end{lemma}

\begin{proof}
From hypothesis, $[B]\circ \alpha=\alpha$, where $[B]=[A_1][A_2]^{-1}$. However, if $[B]$ is not the identity, the equality $[B]\circ \alpha=\alpha$ yields a polynomial equation of degree at most $2$ in $\alpha$ with coefficients in $\F_q$. Therefore, $\alpha\in \F_{q^2}$.
\end{proof}

\subsection{Proof of  Theorem~\ref{thm:main-1}}
We observe that it suffices to prove the theorem in the case that $G$ is a noncyclic group of $\PGL_2(\F_q)$, generated by two elements $[A_1], [A_2]\in \PGL_2(\F_q)$. Let $D_1$ and $D_2$ be the orders of $[A_1]$ and $[A_2]$, respectively. We recall that, for any element $[A]\in \PGL_2(\F_q)$ of order $d$, the $[A]$-invariants have degree two or degree divisible by $d$. In particular, if there exists a monic irreducible polynomial $f\in \F_q[x]$ of degree $n\ge 3$ that is $G$-invariant, then $n$ is divisible by $D_1$ and $D_2$. Therefore, $n=\frac{D_1D_2}{D}\cdot n_0$ for some positive integer $n_0$, where $D=\gcd(D_1, D_2)$. In addition, from Lemmas~\ref{ST4.5} and~\ref{power} we conclude that there exist positive integers $j_1\le D_1$ and $j_2\le D_2$ such that $\gcd(j_1, D_1)=\gcd(j_1, D_2)=1$ and $f$ divides both $F_{A_1^{j_1}, \frac{D_2n_0}{D}}(x)$ and $F_{A_2^{j_2}, \frac{D_1n_0}{D}}(x)$. In other words, if we set $[B_i]=[A_i]^{j_i}$ for $i=1,2$, we have that

\begin{equation}\label{eq:circ}
[B_1]\circ \alpha=\alpha^{q^{D_2n_0/D}}\quad\text{and}\quad [B_2]\circ \alpha=\alpha^{q^{D_1n_0/D}},
\end{equation}
for any root $\alpha\in \overline{\F}_q\setminus \F_{q^2}$ of $f$. In particular, we have the following equalities
$$[B_1B_2]=([B_1]\circ \alpha)^{q^{D_1n_0/D}}=\alpha^{q^{(D_1+D_2)n_0/D}}=([B_2]\circ \alpha)^{q^{D_2n_0/D}}=[B_2B_1]\circ \alpha.$$
In addition, from Eq.~\eqref{eq:circ}, we have that
$$[B_1]^{D_1/D}\circ \alpha=\alpha^{q^{D_1D_2n_0/D}}=[B_2]^{D_2/D}\circ \alpha.$$
Since $\gcd(j_1, D_1)=\gcd(j_2, D_2)=1$, the elements $[B_1]$ and $[B_2]$ have orders $D_1$ and $D_2$, respectively, and they also generate $G$. In addition, since $\alpha$ is not in $\F_{q^2}$, the previous equalities and Lemma~\ref{lem:aux-divisor} entail that $[B_1]\cdot [B_2]=[B_2]\cdot [B_1]$ and $[B_1]^{D_1/D}=[B_2]^{D_2/D}$. Now, the proof of Theorem~\ref{thm:main-1} follows from the following result.

\begin{prop}
Let $G$ be an abelian subgroup of $\PGL_2(\F_q)$, generated by two elements $g_1, g_2$ of orders $d_1$ and $d_2$. If $d=\gcd(d_1, d_2)$ and $g_1^{d_1/d}=g_2^{d_2/d}$, then $G$ is cyclic.
\end{prop}
\begin{proof}
We observe that the structure of $G$ is not changed up to conjugation by an element of $\PGL_2(\F_q)$. In particular, we can suppose that $g_1$ is in reduced form. We have four cases to consider:
\begin{enumerate}[(i)]
\item If $g_1$ if of type $1$, $g_1=[A(a)]$ for some $a\in \F_q\setminus\{0, 1\}$. It is direct to verify that the centralizer of $g_1$ equals the group $\{[A(b)];\, b\in \F_q^*\}$. In particular, $g_2=[A(b)]$ for some $b\in \F_q^*$. Since $\F_q^*$ is cyclic, the subgroup of $\F_q^*$ generated by $a$ and $b$ is also cyclic. If $\theta$ is any generator for such a group, it follows that $G$ is generated by $g_3=[A(\theta)]$.
\item If $g_1$ is of type $2$, $g_1=[\mathcal E]$ has order $p=d_1$. In particular, from Lemma~\ref{lem:order}, we have that either $d_2=p$ or $d_2$ is not divisible by $p$. In particular, either $d_2=p$ or $d=1$. If $d_2=p$, we have that $d=p$ and so $g_1=g_2$. Hence, $G$ is generated by $g_1$. If $d=1$, we have that the orders of $g_1$ and $g_2$ are relatively prime and so $G$ is generated by $g_3=g_1g_2$.
\item If $g_1$ is of type $3$, $g_1$ has order $d_1=2$. If $d_2$ is even, we have that $d=2$ and so $g_1=g_2^{d_2/2}$, hence $G$ is generated by $g_1$. If $d_2$ is odd, the orders of $g_1$ and $g_2$ are relatively prime and so $G$ is generated by $g_3=g_1g_2$.
\item If $g_1$ is of type $4$, $g_1=[D(c)]$ for some $c\in \F_q$ such that $x^2-x-c$ is irreducible over $\F_q$. It is direct to verify that the centralizer of $g_1$ equals the group $G_0=\{[A_t];\, t\in \F_q\}\cup\{[I]\}$, where $A_t=t\cdot I+D(c)$. In particular, $g_1, g_2\in G_0$ and so $G$ is a subgroup of $G_0$. We observe that $D(c)^2=D(c)+cI$ and then, for any $s, t\in \F_q$, the following holds:
$$[A_s]\cdot [A_t]=\begin{cases}[I]&\text{if}\;{s+t+1=0,}\\ 
[A_{h(s, t)}], h(s, t)=\frac{st+c}{s+t+1}&\text{otherwise}.\end{cases}$$ 
 However, since $x^2-x-c\in \F_q[x]$ is irreducible, $K=\F_q[x]/(x^2-x-c)$ is the finite field with $q^2$ elements. Let $k\alpha+w$ be a primitive element of $K=\F_{q^2}$, where $\alpha$ is a root of $x^2-x-c$ and $k, w\in \F_q$ with $k\ne 0$. We claim that, in this case, $G_0$ is the cyclic group generated by $[A_{w/k}]$. In fact, we have $[A_{w/k}]\in G_0$ and, for any $s\in \F_q$, $\alpha+s=(k\alpha+w)^r$ for some $r\in \N$. Hence $$(kx+w)^r\equiv x+s\pmod{x^2-x-c},$$ and then $(kA+wI)^r=A+sI=A_s$. Taking equivalence classes in $\PGL_2(\F_q)$ we obtain $[A_{w/k}]^r=[A_s]$, hence $G_0$ is cyclic. Since $G$ is a subgroup of $G_0$, it follows that $G$ is also cyclic. 
\end{enumerate}
\end{proof}

\subsection{A remark on quadratic invariants}
We observe that, if $f\in \F_q[x]$ is a quadratic irreducible polynomial and $H$ is a subgroup of $\PGL_2(\F_q)$ generated by elements $[B_1], \ldots, [B_s]$, $f$ is $H$-invariant if and only if $B_i\circ f=\lambda_i \cdot f$ for some $\lambda_i\in \F_q$; if we write $f(x)=x^2+ax+b$, the equalities $B_i\circ f=\lambda_i\cdot f$ yield a system of $s$ equations, where $a, b$ and the $\lambda_i$'s are variables. The coefficients of these equations arise from the entries of each $B_i$. Given the elements $B_i$, we can easily discuss the solutions of this system. We remark that we can actually have monic irreducible quadractic polynomials that are invariant by noncyclic groups. For instance, if $q$ is odd and $b\in \F_q^{*}$ is not a square, the polynomial $f=x^2-b$ is invariant by the (noncyclic) group
generated by the elements
$$H_1=\left(\begin{matrix}
-1&0\\ 0&1
\end{matrix}\right)\quad\text{and}\quad H_2=\left(\begin{matrix}
0&1\\ b&0
\end{matrix}\right).$$
For $q$ even, if $c\in \F_q$ is such that $x^2+x+c$ is irreducible over $\F_q$, this polynomial is invariant by the (noncyclic) group generated by the elements
$$H_1=\left(\begin{matrix}
1&0\\1&1
\end{matrix}\right)\quad\text{and}\quad H_2=\left(\begin{matrix}
0&1\\ c&1
\end{matrix}\right).$$

\section{On the number of $[A]$-invariants}
In this section we provide complete enumeration formulas for the number of $[A]$-invariants of degree $n>2$, proving Theorem~\ref{thm:main-2}. We naturally study the number of $[A]$-invariants according to the type of $A$. From Theorem~\ref{thm:conjugates-main}, it suffices to consider elements of $\PGL_2(\F_q)$ of type $1\le t\le 4$ in reduced form. We start with elements of type $1$ and $2$. If $A$ has type $t=1$ (resp. $t=2$) and is in reduced form, we see that $A\circ f$ corresponds to $f(ax)$ (resp. $f(x+1)$). From definition, $[A]\circ f=f$ if and only if $A\circ f=\lambda\cdot f$ for some $\lambda\in \F_q^*$. A comparison on the leading coefficient (resp. the constant term) of the equality $f(x+1)=\lambda\cdot f(x)$ (resp. $f(ax)=\lambda\cdot f(x)$), entails the following fact:

\begin{center}``If $A$ is of type $1$ or $2$ in reduced form and $f\in \F_q[x]$ is a monic irreducible polynomial of degree $k\ge 2$, then $[A]\circ f=f$ if and only if $A\circ f=f$.''\end{center}
Therefore, we are looking for the monic irreducible polynomials $f\in \F_q[x]$ of degree $n$ that satisfies $f(x)=f(ax)$ or $f(x)=f(x+1)$. The number of monic irreducible polynomials satisfying such identities was provided in Theorems~2 and~4 of~\cite{Gar11}. Combining these theorems with Theorem~\ref{thm:conjugates-main}, we easily obtain the following result.

\begin{lemma}
Let $[A]\in \PGL_2(\F_q)$ be an element of type $t\le 2$ and order $D$. Then, for any integer $n>2$, the number $\n_{A}(n)$ of $[A]$-invariants of degree $n$ is \emph{zero} if $n$ is not divisible by $D$ and, for $n=Dm$ with $m\in \N$, the following holds:
\begin{equation*}\n_{A}(Dm)=\frac{\varphi(D)}{Dm}\sum_{d|m\atop \gcd(d, D)=1}\mu(d)(q^{m/d}-\varepsilon),\end{equation*}
where $\varepsilon=1$ if $t=1$ and $\varepsilon=0, D=p$ if $t=2$.
\end{lemma}

Recall that an element of type $3$ in reduced form equals $C(b)$ for some non square $b\in \F_q^*$ and $[C(b)]$ has order two. We observe that, from definition, $[C(b)]\circ f=f$ if and only if $x^{2m}f\left(\frac{b}{x}\right)=\lambda\cdot f(x)$, 
for some $\lambda\in \F_q^*$. Since $b$ is not a square in $\F_q$, the polynomial $x^2-b$ is irreducible over $\F_q$. In particular, if $\theta\in \F_{q^2}$ is a root of $x^2-b$, evaluating both sides of the previous equality at $x=\theta$, we obtain $b^mf(\theta)=\lambda f(\theta)$.
If $f$ is monic irreducible and has degree $2m\ge 3$, $f$ is not divisible by $x^2-b$ and so $f(\theta)\ne 0$. Therefore, $\lambda=b^m$, i.e., $f$ is $[C(b)]$-invariant if and only if
$x^{2m}f\left(\frac{b}{x}\right)=b^m\cdot f(x)$. The number of monic irreducible polynomials satisfying the previous identity was obtained in Corollary~7 of~\cite{MP17}. Combining this corollary with Theorem~\ref{thm:conjugates-main}, we easily obtain the following result.

\begin{lemma}
Let $[A]\in \PGL_2(\F_q)$ be an element of type $3$. In particular, its order is $D=2$. Then, for any integer $n>2$, the number $\n_{A}(n)$ of $[A]$-invariants of degree $n$ is \emph{zero} if $n$ is not divisible by $D$ (i.e., $n$ is odd) and, for $n=2m$ with $m\in \N$, the following holds:
\begin{equation*}\n_{A}(2m)=\frac{1}{2m}\left(-1+\sum_{d|m\atop \gcd(d, 2)=1}\mu(d)q^{m/d}\right).\end{equation*}
\end{lemma}
In particular, cases $1$ and $2$ of Theorem~\ref{thm:main-2} are now proved.

\subsection{Elements of type 4}
Here we establish the last case of Theorem~\ref{thm:main-1}, that corresponds to elements of type $4$. Again, we only consider elements of type $4$ in reduced form. We emphasize that the previous enumeration formulas for elements of type $t\le 3$ are based in the \emph{Mobius Inversion Formula} and its generalizations. This inversion formula is often employed when considering the enumeration of irreducible polynomials with specified properties. We recall a nice generalization of this result.

\begin{theorem}\label{mobius}
Let $\chi:\mathbb N\to \mathbb C$ be a completely multiplicative function (which is, in other words, an homomorphism between the monoids $(\mathbb N, +)$ and $(\mathbb C, \cdot)$). Also let $\mathcal L, \mathcal K:\mathbb N\to \mathbb C$ be two functions such that
$$\mathcal{L}(n)=\sum_{d|n}\chi(d)\cdot \mathcal{K}\left(\frac{n}{d}\right), n\in \mathbb N.$$
Then,
$$\mathcal{K}(n)=\sum_{d|n}\chi(d)\cdot\mu(d)\cdot \mathcal{L}\left(\frac{n}{d}\right), n\in \mathbb N.$$
\end{theorem}
An interesting class of completely multiplicative functions is the class of {\em Dirichlet Characters} and, for instance, the {\em principal Dirichlet character modulo $d$} is the function $\chi_{d}:\mathbb N\to \mathbb N$ such that $\chi_d(n)=1$ if $\gcd(d, n)=1$ and $\chi_d(n)=0$, otherwise. We first present a direct consequence of the results contained in~Subsection~\ref{subsec:lemma-aux}. 
\begin{lemma}\label{lem:type-4-enum-1}
Let $A$ be an element of $\GL_2(\F_q)$ and let $D$ be the order of $[A]$ in $\PGL_2(\F_q)$. Then, for any $m\in \N$, the $[A]$-invariants of degree $Dm>2$  are exactly the irreducible factors of degree $Dm$ of $F_{A^j, m}$, where $j$ runs over the positive integers $\le D-1$ such that $\gcd(j, D)=1$.
\end{lemma}

\begin{proof}
According to Lemma~\ref{ST4.5}, the $[A]$-invariants of degree $Dm>2$ are exactly the irreducible factors of degree $Dm$ of $F_{A, \ell\cdot m}$, where $\ell$ runs over the positive integers $\le D-1$ such that $\gcd(\ell, D)=1$. Additionally, according to Lemma~\ref{power}, for each $\ell$, the following holds: if we set $j(\ell)$ as the least positive solution of $j\ell \equiv 1\pmod D$, the irreducible factors of degree $Dm$ of $F_{A, \ell\cdot m}$ are exactly the irreducible factors of degree $Dm$ of $F_{A^{j(\ell)}, m}$. Clearly $j(\ell)$ runs over the positive integers $j\le D-1$ such that $\gcd(j, D)=1$ (that is, $j(\ell)$ is a permutation of the numbers $\ell$).
\end{proof}

Now, it suffices to count the irreducible polynomials of degree $Dm$ that divide the polynomials $F_{A^j, m}$ for $j\le D-1$ and $\gcd(D, j)=1$. In this case, it is crucial to study the coefficients of $A^j$. When $A$ is an element of type $4$ in reduced form, we can obtain a complete description on the powers of $A$.

\begin{prop}\label{prop:powers-A}
If $A=D(c)$ is an element of type $4$, then \begin{equation}\label{eq:powers-A}A^j=\left( \begin{array}{cc}a_j&b_j\\ c_j & d_j\end{array}\right)=\delta \left( \begin{array}{cc}\alpha^{qj+1}-\alpha^{q+j} & \alpha^{q(j+1)+1}-\alpha^{q+j+1}\\
                \alpha^j-\alpha^{qj}& \alpha^{j+1}-\alpha^{q(j+1)}\end{array}\right), j\in \mathbb Z,\end{equation}
where $\alpha$ is an eigenvalue of $A$ and $\delta=(\alpha-\alpha^q)^{-1}$. In particular, if $D$ is the order of $[A]$, $c_j\ne 0$ for $1\le j\le D-1$.        
\end{prop}

\begin{proof}
Since $A$ is of type $4$, $A$ is a diagonalizable matrix over $\F_{q^2}$ but not over $\F_q$ and we can write
\[
A= M \left( \begin{array}{cc}\alpha&0\\ 0&\alpha^{q}\end{array}\right) M^{-1},
\quad\text{
where}
\quad
M = \left( \begin{array}{cc}\alpha^q& \alpha\\ -1& -1\end{array}\right)
\]
is an invertible matrix and $\alpha$ is an eigenvalue of $A$. From now, Eq.~\eqref{eq:powers-A} follows by direct calculations. We see that $c_j=0$ if and only if $\alpha^j=\alpha^{qj}$. The latter is equivalent to $[A]^j=[I]$ and so $j$ must be divisible by $D$. In particular, for $1\le j\le D-1$, $c_j\ne 0$. 
\end{proof}

From Lemma~\ref{ST4.5}, in general, the irreducible factors of $F_{A^j, m}$ have degree divisible by $D$. The problem relies on counting the irreducible polynomials of degree one and two. From the previous proposition, we describe the linear and quadratic irreducible factors of $F_{A^j, m}$ as follows.

\begin{lemma}\label{lem:type-4-enum-2}
Suppose that $A=D(c)$ is an element of type 4 and order $D$. For any positive integers $j$ and $m$ such that $j\le D-1$ and $\gcd(j, D)=1$, the polynomial $F_{A^j, m}\in \F_q[x]$ has degree $q^m+1$, is free of linear factors and has at most one irreducible factor of degree $2$. In addition, $F_{A^j, m}$ has an irreducible factor of degree $2$ if and only if $m$ is even and, in this case, this irreducible factor is $x^2+c^{-1}x-c^{-1}$.
\end{lemma}

\begin{proof}
From definition, $F_{A^j, m}=b_jx^{q^m+1}-a_jx^{q^m}+d_jx-c_j$. From Proposition~\ref{prop:powers-A}, $c_j\ne 0$ if $1\le j\le D-1$. Therefore, hence $F_{A^j, m}$ has degree $q^m+1$ if $j\le D-1$. We split the proof into cases, considering the linear and quadratic irreducible polynomials.
\begin{enumerate}[(i)]
\item If $F_{A^j, m}$ has a linear factor, there exists $\gamma\in \F_q$ such that $F_{A^j, m}(\gamma)=0$. In this case, $\gamma^q=\gamma$ and a direct calculation yields $F_{A^j, m}(\gamma)=b_j\gamma^2+(d_j-a_j)\gamma-c_j$ and so $\gamma$ is a root of $p_j(x)=b_jx^2+(d_j-a_j)x-c_j$. Let $\alpha, \alpha^q$ be the eigenvalues of $A=D(c)$, hence $\alpha^q+\alpha=1$ and $\alpha^{q+1}=-c$. From Eq.~\eqref{eq:powers-A} and the previous equalities, we can easily deduce that $d_j-a_j=c_j$ and $b_j=cc_j$. Therefore, $p_j(x)$ equals $cx^2+x-1$ (up to a constant). This shows that $\gamma$ is a root of $x^2+c^{-1}x-c^{-1}$. However, since $A=D(c)$ is of type $4$, its characteristic polynomial $p(x)=x^2-x-c$ is irreducible over $\F_q$ and so is $x^2p(\frac{1}{x})=x^2+c^{-1}x-c^{-1}$. In particular, $\gamma$ is not an element of $\F_q$.
\item If $F_{A^j, m}$ has an irreducible factor of degree $2$, there exists $\gamma\in \F_{q^2}\setminus \F_q$ such that $F_{A^j, m}(\gamma)=0$. We observe that, in this case, $\gamma^{q^2}=\gamma$. For $m$ even, $F_{A^j, m}(\gamma)=b_j\gamma^2+(d_j-a_j)\gamma-c_j$ and in the same way as before we conclude that $x^2+c^{-1}x-c^{-1}$ is the only quadratic irreducible factor of $F_{A^j, m}$. If $m$ is odd, $\gamma^{q^m}=\gamma^q$ and equality $F_{A^j, m}(\gamma)=0$ yields $b_j\gamma^{q+1}-a_j\gamma^q+d_j\gamma-c_j=0$.
Raising the $q$-th power in the previous equality and observing that $\gamma^{q^2}=\gamma$, we obtain $$b_j\gamma^{q+1}-a_j\gamma+d_j\gamma^q-c_j=0,$$ and so $(\gamma^q-\gamma)(a_j+d_j)=0$. Since $\gamma$ is not in $\F_q$, the last equality implies that $a_j=-d_j$. However, from Eq.~\eqref{eq:powers-A}, we obtain $\alpha^{qj+1}-\alpha^{q+j}=\alpha^{q(j+1)}-\alpha^{j+1}$. Therefore, $(\alpha^q-\alpha)(\alpha^{qj}+\alpha^j)=0$. Recall that, since $A=D(c)$ is of type $4$, $\alpha$ is not in $\F_q$, i.e., $\alpha^q\ne \alpha$. Therefore $\alpha^{qj}=-\alpha^j$ and then $\alpha^{2qj}=\alpha^{2j}$. Again, from Eq.~\eqref{eq:powers-A}, this implies that $[A]^{2j}=1$, hence $2j$ is divisible by $D$. However, since $j$ and $D$ are relatively prime, it follows that $D$ divides $2$. This is a contradiction, since any element of type $4$ has order $D>2$ (see Lemma~\ref{lem:order}).
\end{enumerate}
\end{proof}

All in all, we finally add the enumeration formula for the number of $[A]$-invariants when $[A]$ is of type $4$, completing the proof of Theorem~\ref{thm:main-2}.

\begin{theorem}
Suppose that $A$ is an element of type $4$ and set $D=\ord([A])$. Then $\n_A(n)=0$ if $n$ is not divisible by $D$ and, for $n=Dm$,
$$\n_A(Dm)=\frac{\varphi(D)}{Dm}\sum_{d|m\atop \gcd(d, D)=1}(q^{m/d}+\epsilon(m/d))\mu(d),$$ 
where $\epsilon(s)=(-1)^{s+1}$. 
\end{theorem}

\begin{proof}
From Theorem~\ref{thm:conjugates-main}, we can suppose that $A$ is in the reduced form, i.e., $A=D(c)$ for some $c$ such that $x^2-x-c$ is irreducible over $\F_q$.  For each positive integer $j$ such that $j\le D-1$ and $\gcd(j, D)=1$, let $n(j)$ be the number of irreducible factors of degree $Dm$ of $F_{A^j, m}$. From Lemma~\ref{lem:type-4-enum-1}, it follows that $$N_A(Dm)=\sum_{j\le D-1\atop \gcd(j, D)=1}n(j).$$
Fix $j$ such that $j\le D-1$ and $\gcd(j, D)=1$. According to Lemma~\ref{ST4.5}, the irreducible factors of $F_{A^j, m}$ are of degree $Dm$, of degree $Dk$, where $k$ divides $m$ and $\gcd(\frac{m}{k}, D)=1$ and of degree at most $2$. For each divisor $k$ of $m$ such that $\gcd(\frac{m}{k}, D)=1$, let $P_{k, m}$ be the product of all irreducible factors of degree $Dk$ of $F_{A^j, m}$ and let $\mathcal L(k)$ be the number of such irreducible factors. Also, set $$\varepsilon_m(x)=\gcd(F_{A^j, m}(x), x^2+c^{-1}x-c^{-1}).$$ Therefore, from Lemma~\ref{lem:type-4-enum-2}, we obtain the following identity
$$\frac{F_{A^j, m}}{\varepsilon_m(x)}=\prod_{k|m\atop \gcd(\frac{m}{k}, D)=1}P_{k, m}.$$
From Lemma~\ref{lem:type-4-enum-2}, $F_{A^j, m}$ has degree $q^m+1$ and the degree of $\varepsilon_m(x)$ is either $0$ or $2$, according to whether $m$ is odd or even. In particular, if we set $\epsilon(m)=(-1)^{m+1}$, taking degrees on the last equality we obtain:
$$q^m+1-\deg(\varepsilon_m(x))=q^m+\epsilon(m)=\sum_{k|m\atop \gcd(\frac{m}{k}, D)=1}\mathcal L (k)\cdot (kD)=\sum_{k|m}\mathcal L(k)\cdot (kD)\cdot \chi_{D}\left(\frac{m}{k}\right),$$
where $\chi_D$ is the {\em principal Dirichlet character modulo $D$}.
From Theorem~\ref{mobius}, we obtain
$$\mathcal L(k)\cdot kD=\sum_{d|k}(q^{k/d}+\epsilon(k/d))\cdot \mu(d)\cdot \chi_D(d)$$
for any $k\in \mathbb N$.
Therefore, $$\mathcal L(m)=\frac{1}{Dm}\sum_{d|m}(q^{m/d}+\epsilon(m/d))\cdot \mu(d)\cdot \chi_D(d)=\frac{1}{Dm}\sum_{d|m\atop \gcd(d, D)=1}(q^{m/d}+\epsilon(m/d))\mu(d).$$
From definition, $n(j)=\mathcal L(m)$ and so 
$$\n_A(Dm)=\sum_{j\le D-1\atop \gcd(j, D)=1}n(j)=\varphi(D)\cdot \mathcal L(m)=\frac{\varphi(D)}{Dm}\sum_{d|m\atop \gcd(d, D)=1}(q^{m/d}+\epsilon(m/d))\mu(d).$$
\end{proof}

\subsection{A remark on previous results}
In~\cite{D62}, the author explores the degree distribution of the polynomials $$(ax+b)x^{q^m}-(cx+d).$$ In the context of $[A]$-invariants, Theorem~5 of ~\cite{D62} can be read as follows:

``\emph{If $[A]\in \PGL_2(\F_q)$ has order $D$ and $m\ge 3$, then the number of $[A]$-invariants of degree $Dm$ equals}\\
\begin{equation}\label{eq:d62}\frac{\varphi(D)}{Dm}\sum_{d|m\atop{\gcd(d, D)=1}}\mu(d)q^{m/d}.\text{''}\end{equation}
We observe that Eq.~\eqref{eq:d62} disagrees with Theorem~\ref{thm:main-2} and turns out to be incorrent in many cases. For instance, if $q=2$ and $A=\left(\begin{matrix}
0&1\\
1&1
\end{matrix}\right)$, then $A$ is diagonalizable over $\F_{4}\setminus \F_2$ and $[A]$ has order $3$. If $r=3^k$, Theorem~\ref{thm:main-2} entails that $$\n_{A}(3r)=\n_{A}(3^{k+1})=\frac{2}{3^{k+1}}(2^{3^k}+1)\in \mathbb N.$$ 
However, Eq.~\eqref{eq:d62} provides $\n_{A}(3^{k+1})=\frac{2}{3^{k+1}}\cdot 2^{3^k}$, that is not even an integer. The lack of accurarcy in Eq.~\eqref{eq:d62} is, perhaps, due to the miscalculation of the linear and quadratic factors of the polynomials $(ax+b)x^{q^m}-(cx+d)$ in~\cite{D62}.

\begin{center}{\bf Acknowledgments}\end{center}
This work was partially conducted during a visit to Carleton University, supported by the Program CAPES-PDSE (process - 88881.134747/2016-01). The author was partially supported by FAPESP 2018/03038-2, Brazil.

\end{document}